\documentclass[12pt]{article}
\usepackage[utf8]{inputenc}
\usepackage[T1]{fontenc}
\usepackage[english]{babel}
\usepackage{xcolor}
\usepackage[top=2cm, bottom=2cm, left=2cm, right=2cm]{geometry}
\usepackage{layout}
\usepackage{setspace}
\usepackage{charter}
\usepackage{url}
\usepackage{graphicx}
\usepackage{tikz}
\usetikzlibrary{matrix,arrows}
\usepackage{amsmath}
\usepackage{amssymb}
\usepackage{mathrsfs}
\usepackage{amsthm}
\usepackage{hyperref}
\usepackage{float}
\usepackage[width = 1 \textwidth]{caption}
\usepackage[format=plain,
            textfont=it]{caption}
\usepackage{titlesec}
\titleformat{\paragraph}
{\normalfont\normalsize\bfseries}{\theparagraph}{1em}{}
\titlespacing*{\paragraph}
{0pt}{3.25ex plus 1ex minus .2ex}{1.5ex plus .2ex}
\usepackage{enumitem}
\setlist{itemsep = 3 pt}
\usepackage{mathtools}

\newcommand\customeq{\stackrel{\mathclap{\normalfont\scriptsize\mbox{n}}}{=}}
\newcommand\customleq{\stackrel{\mathclap{\normalfont\scriptsize\mbox{n}}}{\leq}}
\newcommand\customgeq{\stackrel{\mathclap{\normalfont\scriptsize\mbox{n}}}{\geq}}

\newcommand\customleqdeux{\stackrel{\mathclap{\normalfont\scriptsize\mbox{2}}}{\leq}}

\newcommand\customeqtrois{\stackrel{\mathclap{\normalfont\scriptsize\mbox{3}}}{=}}
\newcommand\customgeqtrois{\stackrel{\mathclap{\normalfont\scriptsize\mbox{3}}}{\geq}}
\newcommand\customleqtrois{\stackrel{\mathclap{\normalfont\scriptsize\mbox{3}}}{\leq}}

\newcommand\customleqcinq{\stackrel{\mathclap{\normalfont\scriptsize\mbox{5}}}{\leq}}

\newcommand{\R}{\mathbb{R}}
\newcommand{\C}{\mathbb{C}}
\renewcommand{\P}{\mathbb{P}}
\newcommand{\Z}{\mathbb{Z}}
\newcommand{\N}{\mathbb{N}}

\makeatletter
\newcommand{\specialcell}[1]{\ifmeasuring@#1\else\omit$\displaystyle#1$\ignorespaces\fi}
\makeatother

\theoremstyle{definition}
\newtheorem{lem}{Lemma}[section]
\newtheorem{prop}{Proposition}[section]
\newtheorem{defin}{Definition}[section]
\newtheorem{nota}{Notation}[section]
\newtheorem{thm}{Theorem}[section]
\newtheorem*{thm*}{Theorem}
\newtheorem{cor}{Corollary}[section]
\newtheorem{rmk}{Remark}[section]

\let\oldbibliography\thebibliography
\renewcommand{\thebibliography}[1]{\oldbibliography{#1}
\setlength{\itemsep}{-0pt}} 

\title{Real critical points of $T$-polynomials that are sums of squared monomials and topology of $T$-hypersurfaces}
\author{Aloïs Demory}
\date{}

\begin{document}

\maketitle

\tableofcontents

\section{Introduction}

Since its beginning with the foundational works of A. Harnack (\cite{harnack}) and D. Hilbert (\cite{Hilbert1933}) on real algebraic curves in $\R \P^2$, the study of topology of real algebraic hypersurfaces has been divided into two complementary directions: finding restrictions on the topology of real
algebraic hypersurfaces and constructing hypersurfaces that are extremal with respect to these restrictions.

\begin{defin}
    A \textit{real algebraic hypersurface} $A$ of degree $m$ in $\P^n$ is given by the datum of polynomial $P$ of degree $m$ in $n+1$ variables with real coefficient up to multiplication by a non-zero real constant. The \textit{real part} $\R A$ of $A$ is the compactification in $\R \P^n$ of $\{x \in (\R^*)^n \; | \; P(x)=0  \}$, while its \textit{complex part} $\C A$ is the compactification in $\C \P^n$ of $\{z \in (\C^*)^n \; | \; P(z)=0  \}$.  The real algebraic hypersurface $A$ is said to be \textit{non-singular} if $\mathbb{C}A$ is a smooth complex manifold.
\end{defin}
 
Combinatorial patchworking provides a way of constructing non-singular real algebraic hypersurfaces starting from a purely combinatorial datum: a triangulation with integer vertices of a polytope and a sign distribution on the vertices of this triangulation. The hypersurfaces that are obtained using this method are called $T$\textit{-hypersurfaces}. Introduced by O. Viro in the early 1980s (\cite{Vir83}), combinatorial patchworking has proven to be extremely fruitful and allowed for many new results (see \textit{e.g.} \cite{tsurf}, \cite{itenbergviro2007}). In particular, I. Itenberg and Viro have established in \cite{fakeitenbergviro} that the following inequality, which is probably the most important restriction regarding topology of real algebraic varieties, is sharp for hypersurfaces of any degree in projective spaces of any dimension.

\begin{thm*}
    \label{smith}
    \textbf{(Smith-Thom inequality, see \textit{e.g.} \cite{floyd}, \cite{bredon})}
    Let $A$ be an $n$-dimensional real algebraic variety. Then $\sum_{i=0}^{n} dim_{\mathbb{Z}_2}H_i(\mathbb{R}A;\mathbb{Z}_2) \leq \sum_{i=0}^{2n} dim_{\mathbb{Z}_2}H_i(\mathbb{C}A;\mathbb{Z}_2)$.
\end{thm*}

Most of the constructions that have been carried out so far using combinatorial patchworking take as starting datum a \textit{primitive} triangulation, that is to say a triangulation whose top-dimensional simplices all have the smallest possible Euclidean volume. It turns out that the real parts of the real algebraic hypersurfaces that can be obtained this way have very specific topological properties, such as the following ones. Here and in the rest of the text, given a topological manifold $X$, we denote by $\chi(X)$ its Euler characteristic, by $\sigma(X)$ its signature when it is well-defined and by $b_i(X)$ the number $\mathrm{dim}_{\Z_2} H_i(X;\Z_2)$ for any non-negative integer $i$.

\begin{thm*}
\label{thmbertrand}
\textbf{(\cite{bertrand_2010})} If $A$ is an even-dimensional real algebraic hypersurface obtained by primitive combinatorial patchworking in a non-singular toric variety, then $\chi(\mathbb{R}A) = \sigma(\mathbb{C}A)$.
\end{thm*}

\begin{thm*}
\textbf{(\cite{renaudineau})} If $A$ is a real algebraic hypersurface in $\mathbb{P}^n$ obtained by primitive combinatorial patchorking in $\mathbb{P}^n$, then
\begin{itemize}
    \item if $n-1$ is even, then $b_{\frac{n-1}{2}}(\mathbb{R}A) \leq h^{\frac{n-1}{2},\frac{n-1}{2}}(\mathbb{C}A)$;
    \item for every integer $i$ such that $0\leq i \leq n-1$ and $i \neq \frac{n-1}{2}$, one has $b_i(\mathbb{R}A) \leq h^{i,n-1-i}(\mathbb{C}A)+1$.
\end{itemize}
\end{thm*}

On the other hand, very little is known about the topology of real algebraic hypersurfaces obtained via non-primitive patchworking, especially in high dimensions. In the case of surfaces in $\P^3$, Itenberg and E. Shustin were able to prove the following bounds.

\begin{thm*}
    \textbf{(\cite{itenbergshustincritpts})} For any real algebraic surface $A$ of degree $m$ in $\P^3$ obtained by combinatorial patchworking, one has
    \begin{equation*}
        b_0(\R A) \customleqtrois \frac{16}{39}m^3 \; \mathrm{and} \; b_1(\R A) \customleqtrois \frac{7}{9}m^3.
    \end{equation*}
\end{thm*}

Here and in the rest of the paper, $X \customleq Y$ (respectively, $X \customgeq Y$) means that $X \leq Y + R(m)$ (respectively, $X \geq Y + R(m)$), where $R \in \R[m]$ is a polynomial of degree at most $n-1$ whose coefficients do not depend on $m$. We write $X \customeq Y$ if $X \customleq Y$ and $X \customgeq Y$.

The present work constitutes an attempt at understanding non-primitive combinatorial patchworking by focusing on hypersurfaces that are obtained by using combinatorial patchworking starting with triangulations that are dilations by $2$ of other triangulations. Such hypersurfaces will be called $T^2$\textit{-hypersurfaces}. The real part of any $T^2$-hypersurface in $\P^n$ is symmetric with respect to all coordinate hyperplanes. Hence, in a sense, all the topological information about the real part of the hypersurface is confined to the positive orthant, which allows for a simpler study. Moreover, the study of such hypersurfaces seems to be relevant for the applications: for example, hypersurfaces in the positive orthant that are obtained by combinatorial patchworking have recently be proven to appear as decision boundaries of binary classification neural networks (see \cite{machinelearningduportugal}).

In this paper, we use the tools developed by Itenberg and Shustin in \cite{itenbergshustincritpts} to study the critical points of polynomials defining patchworked hypersurfaces and to prove the following results.

\begin{thm*}
    \textbf{\ref{thm_partial_sums_betti_numbers}}
    Let $A$ be a $T^2$-hypersurface of degree $m$ in $\P^n$.
    For any positive integer $k \leq n-1$, one has
    \begin{equation*}
    \sum_{i=0}^{k-1} \frac{(k+n-1-i)!}{(k-1-i)!n!} b_i(\R A) \customleq \sum_{i=0}^{k-1} \frac{(k+n-1-i)!}{(k-1-i)!n!} h^{i,n-1-i}(\C A).\end{equation*}
    
    Additionally, one has $b_1(\R A) \customleq h^{1,n-2}(\C A)$.

    Moreover, if $n \leq 6$, then for any non-negative integer $k\leq n-1$, one has
    \begin{equation*}
        \sum_{i=0}^k { n-i \choose n-k } b_i(\R A) \customleq \sum_{i=0}^k { n-i \choose n-k } h^{i,n-1-i}(\C A).
    \end{equation*}
\end{thm*}

\begin{thm*}
    \textbf{\ref{main_theorem}}
    Let $n >2$ and let $\mathcal{F} = (A_m)_{m \in 2\mathbb{N}}$ be a family of $T^2$-hypersurfaces in $\P^n$ indexed by their degree. Then there exists a polynomial $Q_{\mathcal{F}} \in \R[m]$ of degree $n$ with positive leading coefficent such that, for any $m \in 2 \N$, one has
    \begin{equation*}
        \sum_{i = 1}^{n-1} b_i(\R A_m; \Z_2) \customleq m^n - Q_{\mathcal{F}}(m) \customeq \sum_{i = 1}^{2n-2} b_i(\C A_m) - Q_{\mathcal{F}}(m).
    \end{equation*}
\end{thm*}

We also make some constructions of $T^2$-hypersurfaces in order to evaluate the sharpness of the proven bounds. We obtain for example the following result.

\begin{thm*}
    \textbf{\ref{prop_large_b0_alldim}}
    Let $n$ be a positive integer.
    Let $\Pi^n \subset \R^n$ be the polytope with vertices $(0,...,0)$, $(2,1,...,1)$, $(1,2,1,...,1)$, ..., $(1,...,1,2)$. There exists a family $(A_m^n)_{m \in 2\N}$ of $T^2$-hypersurfaces in the toric variety associated to $\Pi^n$ such that $b_0(\R A_m^n) \customeq \frac{n}{n+1}h^{n-1,0}(\C A_m^n)$ and such that for any even positive integer $m$, the hypersurface $A_m$ has Newton polygon $m\Pi^n$.
\end{thm*}

In Section \ref{section_real_crit_pts}, we recall some facts about Morse inequalities, real critical points of polynomials obtained through combinatorial patchworking and hypersurfaces that are defined by sums of squared monomials. Then, in Section \ref{section_triangulations}, we prove some preliminary combinatorial lemmas. The proofs of Theorems \ref{thm_partial_sums_betti_numbers} and \ref{main_theorem} are presented in Section \ref{section_proof_bounds} and constructions, including the proof of Theorem \ref{prop_large_b0_alldim}, are carried out in Section \ref{section_constructions}.

\section{Preliminaries}

\label{section_real_crit_pts}

\subsection{Morse inequalities for real algebraic hypersurfaces }

\label{section_morse_ineq}

\begin{sloppypar}
    Let $P$ be the product of a real polynomial of degree $m$ in $n$ variables and any monomial in $n$ variables. Let $A$ be the real algebraic hypersurface in $\P^n$ defined by $P$.
\end{sloppypar}

\begin{sloppypar}
    Assume that $P$ has only non-degenerate critical points in $(\mathbb{C}^*)^n$ and that the hypersurface $A$ is non-singular. Denote by $c_i^+(P)$ (respectively, $c_i^-(P)$) the number of real critical points of $P$ in ($\mathbb{R}^*)^n$ of index $i$ and with positive (respectively, negative) critical value. For any $i \in \{ 0,...,n-1 \}$, put $c_i(P) = c_i^+(P) + c_i^-(P)$.
\end{sloppypar}

\begin{prop}
    \textbf{(see \textit{e. g.} \cite{itenbergshustincritpts})}
    For any $i \in \{0,...,n-1 \}$, one has 
    \begin{equation*}
        b_i(\mathbb{R}A) \customleq \mathrm{min}(c_i^-(P) + c_{n-i}^+(P), c_{n-i-1}^-(P) + c_{i+1}^+(P)).
    \end{equation*}
\end{prop}

\subsection{Viro's combinatorial patchworking theorem}

\begin{sloppypar}
    Let $n$, $m$ be positive integers. Let $\Delta_m^n$ be the $n$-dimensional simplex in $\mathbb{R}^n$ with vertices $(0,...,0)$, $(m,0,...,0)$, $(0,m,0,...,0)$, ..., $(0,...,0,m)$. We denote by $(\Delta_m^n)^*$ the union of the images of $\Delta_m^n$ under all compositions of reflections with respect to coordinate hyperplanes.
    Let $\tau$ be a rectilinear triangulation with integer vertices of $\Delta_m^n$. Let $\mu: V(\tau) \rightarrow \{ +,- \}$ be a sign distribution on the set $V(\tau)$ of vertices of $\tau$. Extend $\tau$ to a symmetric triangulation $\tau^*$ of $(\Delta_m^n)^*$. 
    Extend the sign distribution $\mu$ to a sign distribution $\mu^*$ at the vertices of $\tau^*$ using the following rule : if two vertices are the images of each other under a reflection with respect to a certain coordinate hyperplane, then their signs are the same if their distance to the hyperplane is even, and their signs are opposite otherwise.
    If an $n$-simplex of $\tau^*$ has vertices with different signs, consider the convex hull of the middle points of its edges that have vertices of opposite sign. It is a piece of hyperplane. Denote by $\Gamma$ the union of all such hyperplane pieces. It is a piecewise-linear hypersurface contained in $(\Delta_m^n)^*$.
    Identify each pair of antipodal points on the boundary of $(\Delta_m^n)^*$ to obtain a piecewise-linear manifold $\widetilde{(\Delta_m^n)^*}$ and denote by $\widetilde{\Gamma}$ (respectively, $\widetilde{\tau^*}$) the image of $\Gamma$ (respectively, $\tau^*$) under this identification.
    The triangulation $\tau$ is said to be \textit{convex} if there exists a piecewise-linear convex function $\nu:\Delta_m^n \rightarrow \mathbb{R}$ such that $\nu(V(\tau)) \subset \mathbb{Z}$ and the domains of linearity of $\nu$ coincide with the $n$-simplices of $\tau$. The following result is due to Viro.
\end{sloppypar}

\begin{thm}
\textbf{(\cite{Vir83}, see also \cite{viropatchwork})} If $\tau$ is convex, then there exists a non-singular real algebraic hypersurface $A$ in $\mathbb{P}^n$ and a homeomorphism $\mathbb{RP}^n \rightarrow (\widetilde{(\Delta_m^n)^*})$ mapping $\mathbb{R}A$ onto $\widetilde{\Gamma}$. Such a hypersurface $A$ can be taken to be the compactification of the variety defined by a polynomial with Newton polytope $\Delta_m^n$ and of the form 
\begin{center}
    $P_t(x_1,...,x_n) = \sum_{(i_1,...,i_n) \in V(\tau)} \mu(i_1,...,i_n)x_1^{i_1}...x_n^{i_n}t^{\nu (i_1,...,i_n)}$    
\end{center}
for some positive and sufficiently small real number $t$.
\end{thm}

\begin{rmk}
    Viro's theorem remains true if one replaces $\Delta_m^n$ with another $n$-dimensional lattice polytope $\Pi \subset (\R_{\geq 0})^n$ and $\P^n$ with the toric variety $X_{\Pi}$ associated to $\Pi$. In that more general case, the hypersurface $A$ is Newton-non-degenerate but is not necessarily non-singular when $X_{\Pi}$ is not non-singular.
\end{rmk}

\begin{defin}
    For $t$ as above, the polynomial $P_t$ is said to be a $T$\textit{-polynomial} of degree $m$ in $n$ variables. We denote by $\tau(P_t)$ the triangulation of $\Delta_m^n$ defined by $\nu$.
\end{defin}

There are two classes of triangulations that are often used to construct real algebraic hypersurfaces using combinatorial patchworking.

\begin{defin}
    A triangulation with integer vertices of an $n$-dimensional lattice polytope in $\R^n$ is said to be \textit{primitive} if the Euclidean volume of every $n$-dimensional simplex is equal to $\frac{1}{n!}$.
\end{defin}

\begin{defin}
    A triangulation $\tau$ with integer vertices of an $n$-dimensional lattice polytope $\Pi$ in $\R^n$ is said to be \textit{maximal} if all integer points of $\Pi$ are vertices of $\tau$. 
\end{defin}

\subsection{Real critical points of $T$-polynomials}

\begin{sloppypar}
    Let $P$ be a $T$-polynomial of degree $m$ in $n$ variables. Let $O$ be an integer point in $\R^n \setminus \Delta_m^n$ that is not contained in any hyperplane that contains an $(n-1)$-simplex of $\tau(P)$. Such a point is said to be \textit{generic with respect to} $\tau(P)$. If $O$ is generic with respect to $\tau(P)$, then the function $x^{-O}P$ has only non-degenerate critical points (see \cite{itenbergshustincritpts}). Let $\sigma^n$ be an $n$-simplex of $\tau(P)$.
\end{sloppypar}

\begin{defin}
    An $(n-1)$-face $\sigma^{n-1}$ of $\sigma^n$ is said to be $O$\textit{-visible} (respectively, $O$\textit{-non-visible}) if the cone of vertex $O$ over $\sigma^{n-1}$ does not intersect (respectively, intersects) the interior of $\sigma^n$.
\end{defin}

\begin{defin}
    The $O$-index of $\sigma^n$, denoted by $i^O(\sigma^n)$, is the number of $O$-visible faces of $\sigma^n$. The $O$-co-index of $\sigma^n$ is the number $n-i^O(\sigma^n)$.
\end{defin}

\begin{defin}
    The $O$-root (respectively, $O$-co-root) of $\sigma^n$, denoted by $V_+^O(\sigma^n)$ (respectively, $V_-^O(\sigma^n)$) is the set of vertices of $\sigma^n$ which belong to all $O$-visibles (respectively, all $O$-non-visible) $(n-1)$-faces of $\sigma^n$. 
\end{defin}

\begin{defin}
    Let $s$ be a composition of reflections with respect to coordinate hyperplanes of $\mathbb{R}^n$. The symmetric copy $s(\sigma^n)$ of $\sigma^n$ is said to be $\textit{real}$ $O$\textit{-critical} if the vertices of $s(\sigma^n)$ that are symmetric copies of the vertices in $V_+^O(\sigma^n)$ bear the same sign and the vertices of $s(\sigma^n)$ that are symmetric copies of the vertices in $V_-^O(\sigma^n)$ bear the opposite sign.
\end{defin}

\begin{thm}
    \textbf{(\cite{itenbergshustincritpts})} \label{thmitenbergshustin} There exists a one-to-one correspondence between the real critical points in $(\R^*)^n$ of the function $x^{-O} P$ and the real $O$-critical symmetric copies of simplices of $\tau(P)$ such that the index of a real critical point of $x^{-O}P$ with positive (respectively, negative)  critical value is equal to the $O$-index (respectively, $O$-co-index) of the corresponding simplex.
\end{thm}

\begin{rmk}
    Notice that for any choice of $O$, the pair $((\R^*)^n,\{x \in (\R^*)^n\;|\; x^{-O} P(x) = 0\})$ is homeomorphic to $((\R^*)^n,\{x \in (\R^*)^n\;|\;P(x) = 0\})$.
\end{rmk}

\subsection{$T^2$-polynomials}

\begin{nota}
    We put $TP^n = \{ P| \;P$ is a $T$-polynomial in $n$ variables; $\tau(P)$ is the dilation by $2$ of a triangulation with integer vertices of $\Delta_{\frac{m}{2}}^n$, where $m$ is the degree of $P \}$. An element $P$ of $TP^n$ is said to be a $T^2$\textit{-polynomial} in $n$ variables. The real algebraic hypersurface in $\P^n$ defined by $P$ will be called a $T^2$\textit{-hypersurface}.
\end{nota}

\begin{nota}
    Let $a$, $b \in [0,1]$ be two real numbers such that $a \leq b$. Let $Q$, $R \in \mathbb{R}[m]$ be two polynomials of degree at most $n-1$.\\ 
    We put $TP^n(a,Q,b,R) := \{ P \in TP^n |\; \frac{a}{n!2^n}m^n + Q(m) \leq \#V(\tau(P)) \leq \frac{b}{n!2^n}m^n + R(m)\}$. Here, given a triangulation $\tau$, we denote by $\#V(\tau)$ the cardinal of the set of vertices of $\tau$.
\end{nota}


\begin{lem}
    \textbf{(\cite{itenbergshustincritpts})} Let $P\in TP^n$ and let $O$ be an integer point that is generic with respect to $\tau(P)$. Any $n$-simplex $\sigma^n$ of the triangulation $\tau(P)$ has either $2^n$ or $0$ real $O$-critical symmetric copies.
\end{lem}

\begin{nota}
    Let $m$ and $n$ be positive integers, let $\tau$ be a convex triangulation with integer vertices of $\Delta_m^n$, and let $O$ be an integer point that is generic with respect to $\tau$. For any $i \in \{1,...,n \}$, we denote by $\mathcal{S}^O_i(\tau)$ the number of $n$-simplices of $\tau$ of $O$-index $i$. Additionnaly, we denote by $\bar{\mathcal{S}}_i^O(\tau)$ the number of $n$-simplices of $\tau$ of $O$-index $i$ that have real critical copies.
\end{nota}

\begin{rmk}
    For any $T^2$-polynomial $P$, for any integer $i \in \{0,...,n-1 \}$, and for any integer point $O$ that is generic with respect to $\tau(P)$, one has $c_i^-(x^{-O}P) + c_{n-i}^+(x^{-O}P) = 2^n \bar{\mathcal{S}}_{n-i}^O(\tau(P)) \leq 2^n \mathcal{S}^O_i(\tau(P))$.
\end{rmk}

\subsection{Sums of squared monomials and sibling hypersurfaces}

\label{section_sibling_hypersurf}

\begin{nota}
    Let $m$ and $n$ be positive integers. We denote by $SQM(m,n)$ the set of polynomials that are of degree $m$ in $n$ variables, that have real coefficients and that are sums of squared monomials. Any element $P$ of $SQM(m,n)$ is said to be an $SQM$\textit{-polynomial of degree} $m$ \textit{in} $n$ \textit{variables}. The real algebraic hypersurface defined by $P$ is said to be an $SQM$-hypersurface.
\end{nota}

\begin{sloppypar}
    Given an $SQM$-polynomial $P$ of degree $m$ in $n$ variables, the real algebraic hypersurface $A$ in $\mathbb{P}^n$ defined by $P$ is symmetric with respect to all coordinate hyperplanes. Compositions of reflections with respect to coordinate hyperplanes commute with the restrictions of the complex conjugation on $\mathbb{CP}^n$ to $\C A$. The composition of the conjugation with any composition of reflections with respect to coordinate hyperplanes gives a new antiholomorphic involution on $\mathbb{CP}^n$, whose fixed point set is homeomorphic to $\mathbb{RP}^n$. Each coordinate of a point in the fixed point set can be either real or pure imaginary, depending on the chosen composition of reflections. Hence we obtain $2^n$ natural real structures on $\mathbb{C}A$. The obtained real algebraic varieties are called the \textit{sibling hypersurfaces} of $A$. The set of the sibling hypersurfaces of $A$ is denoted by $\mathcal{F}_A$.
\end{sloppypar}

\begin{sloppypar}
    The real part of each sibling hypersurface of $A$ can be seen as a gluing of $2^n$ copies of the intersection of $\mathbb{R}B$ with a certain (closed) orthant of $\mathbb{RP}^n$, where $B$ is the real algebraic hypersurface defined by the polynomial obtained from $P$ by dividing all exponents by $2$. Hence, the disjoint union of all the sibling real parts carries all the topological information about $\mathbb{R}B$ and it is possible to obtain some topological restrictions on the disjoint union of the real parts of the sibling hypersurfaces from the known restrictions on $\mathbb{R}B$.
\end{sloppypar}

\begin{prop}
    \textbf{(Application of the Smith-Thom inequality)} 
    \begin{center}
    $\sum_{S \in \mathcal{F}_A} b_*(\mathbb{R}S) \customeq 2^n b_*(\mathbb{R}B) \leq 2^n b_*(\mathbb{C}B)$. \qedsymbol \end{center} 
\end{prop}
\begin{prop}
    \textbf{(Application of the Petrovskii-Oleinik inequality)} If $n-1$ is even and $B$ is non-singular, then 
    \begin{center}$-2^nh^{n-1,n-1} (\mathbb{C}B)+ 2^{n+1} \leq \sum_{S \in \mathcal{F}_A} \chi(\mathbb{R}S) \leq 2^nh^{n-1,n-1} (\mathbb{C}B)$.\qedsymbol \end{center} 
\end{prop}

\begin{sloppypar}
    Every $T^2$-polynomial is an $SQM$-polynomial. If $P$ is a $T^2$-polynomial and $\tau(P)$ is the dilation by $2$ of a primitive triangulation, then one can use the known results on hypersufaces obtained by primitive combinatorial patchworking to get restrictions on the topology of the disjoint union of the siblings of the hypersurface defined by $P$.
\end{sloppypar}

\begin{lem}
    Let $P$ be a $T^2$-polynomial defining a $T^2$-hypersurface $A$ in $\P^n$. Then each sibling hypersurface of $A$ can be defined by a $T^2$-polynomial $P'$ such that:
    \begin{itemize}
        \item $\tau(P) = \tau(P')$;
        \item the signs on the vertices of $\tau(P')$ are the signs on the vertices of the dilation by $\frac{1}{2}$ of the datum in the orthant defined by a certain composition of reflections with respect to coordinate hyperplanes.
    \end{itemize}
\end{lem}

\begin{proof}
    Let us restrict to $(\mathbb{C}^*)^n$. Choose a composition $\delta$ of reflections with respect to coordinate hyperplanes. The antiholomorphic involution $conj \:\circ \delta$ on $(\mathbb{C}^*)^n$ has as real part a subspace of real dimension $n$. 
    The coordinates of the points of this subspace are either real or pure imaginary. 
    We make a change of variables by multiplying the suitable coordinates by $i$. The polynomial expressed in these new variables is again a $T^2$-polynomial, that satisfies the desired conditions.
\end{proof}

\begin{prop}
    \textbf{(Application of \cite{bertrandthese})} If an even-dimensional real algebraic hypersurface $A$ of degree $m$ in $\P^n$ is obtained by combinatorial patchworking starting with a triangulation $\tau$ that is the dilation by $2$ of a primitive triangulation, then $\sum_{S \in \mathcal{F}_A} \chi(S) = 2^n \sigma(\C B)$, where $B$ is a generic real algebraic hypersurface of degree $\frac{m}{2}$ in $\P^n$.
\end{prop}
\begin{prop}
    \textbf{(Application of \cite{renaudineau})} If a real algebraic hypersurface $A$ of degree $m$ in $\P^n$ is obtained by combinatorial patchworking starting with a triangulation $\tau$ that is the dilation by two of a primitive triangulation, then for any integer $i$, one has
    \begin{equation*}
        \sum_{S \in \mathcal{F}_A} b_i(\R S) \customleq h^{i,n-1-i}(\C A).
    \end{equation*}
\end{prop}

\begin{rmk}
    In the case $n=3$, the first statement can be adapted to dilations of maximal triangulations by replacing "$=$" with "$\geq$". It follows from I. Itenberg's result on maximal triangulations (\cite{tsurf}).
\end{rmk}

\section{Triangulations and simplex vision}

\label{section_triangulations}

\subsection{Number of simplices of each dimension}

\subsubsection{General case}

\begin{sloppypar}
    Let $\tau$ be a triangulation with integer vertices of $\Delta_m^n$. We want to compute the number of $i$-simplices of $\tau$ whose relative interiors are contained in the interior of $\Delta_m^n$. This number is denoted by $s_i(\tau)$. In this section, given a polytope $\Pi \subset \R^n$, we denote by $l^*(\Pi)$ the number of integer points that lie in the relative interior of $\Pi$.
\end{sloppypar}

\begin{lem}
    \label{lem_maximal_simplices}
    Let $\sigma$ be an $i$-dimensional simplex with integer vertices in $\mathbb{R}^m$. Suppose that $\sigma$ does not contain any integer point other than its vertices. Then, for any positive integer $k$, one has that $l^*(k \sigma) \geq {k-1 \choose i}$.
\end{lem}

\begin{proof}
    If $i=1$ or $i=2$, the statement is clear, as edges and triangles that correspond to the description above are necessarily primitive. Suppose that the statement is true for every integer $i \leq N$. We prove the statement for $i = N+1$. Choose an $N$-face $\hat{\sigma}$ of $\sigma$ and an edge $\check{\sigma}$ of $\sigma$ that is not an edge of $\hat{\sigma}$. Subdivide $k\sigma$ by $k-1$ hyperplanes parallel to $k\hat{\sigma}$, each passing through a different integer point in the relative interior of $k\check{\sigma}$. The intersection of the interior of $k\sigma$ with the hyperplanes is the disjoint union of translations by integer vectors of $\hat{\sigma}$, $2\hat{\sigma}$, ..., $(k-1)\hat{\sigma}$. Now the result follows from the statement in dimension $N$: we have that 
    \begin{equation*}
        l^*(k\sigma) \geq \sum_{j=1}^{k-1} l^*(j \hat{\sigma}) \geq \sum_{j=1}^{k-1} {j-1 \choose N} = {k-1 \choose N+1}.    
    \end{equation*}
\end{proof}

\begin{rmk}
    The minimum number for each $l^*(k \sigma)$ is realized if $\sigma$ is primitive.
\end{rmk}

\begin{cor}
    \label{corollary_number_simplices_number_points}
    Let $\tau$ be a triangulation with integer vertices of $\Delta_m^n$. Then one has $l^*(k\Delta_m^n) \geq \sum_{i=0}^{k-1} {k-1 \choose i} s_i(\tau)$.
\end{cor}

\begin{proof}
    Any triangulation $\tau$ can be refined into a maximal triangulation $\tau'$, which necessarily contains at least the same number of simplices of each dimension. Each $i$-simplex of the dilation by $k$ of $\tau'$ and whose relative interior is a subset of the interior of $k\Delta_m^n$ contains at least ${k-1 \choose i}$ integer points by Lemma \ref{lem_maximal_simplices}. The result follows.
\end{proof}

\subsubsection{Primitive and maximal triangulations}

\begin{lem}
    Let $\tau$ be a primitive triangulation of $\Delta_m^n$. Let $k$ be an integer such that $1 \leq k \leq n+1$. One has $s_{k-1}(\tau) = \sum_{j=1}^k (-1)^{j+k} {k-1 \choose j-1} l^*(j \Delta_m^n)$.
\end{lem}


\begin{proof}
    One has

\begin{equation*}
    l^*(k\Delta_m^n) = \sum_{i=0}^{k-1} {k-1 \choose i} s_i(\tau),
\end{equation*}

\begin{equation}
    \label{equation_nombre_simplexes}
    s_{k-1}(\tau) = l^*(k\Delta_m^n) - \sum_{i=0}^{k-2} {k-1 \choose i} s_i(\tau).
\end{equation}

    Suppose that for any $i\leq k-2$, we have $s_i(\tau) = (-1)^i\sum_{j=0}^i (-1)^j {i \choose j} l^*((j+1)\Delta_m^n)$ (this is clearly true if $i=0$). We want to express $s_{k-1}(\tau)$ as a linear combination of $l^*(\Delta_m^n)$, $l^*(2\Delta_m^n)$, ..., $l^*(k\Delta_m^n)$. If $j \neq k$, the coefficient in front of $l^*(j\Delta_m^n)$ in equation (\ref{equation_nombre_simplexes}) is equal to
    \begin{equation*}
        - \sum_{i = j-1}^{k-2} (-1)^i (-1)^{j-1} {k-1 \choose i} {i \choose j-1} = (-1)^{j+k} {k-1 \choose j-1}.
    \end{equation*}
    
    Hence, $s_{k-1}(\tau) = \sum_{j=1}^k (-1)^{j+k} {k-1 \choose j-1} l^*(j \Delta_m^n)$.
\end{proof}


\begin{rmk}
    If $\tau$ is a maximal triangulation, we have $s_0(\tau) = l^*(\Delta_m^n)$ and $l^*(2\Delta_m^n) \geq s_0(\tau) + s_1(\tau) = l^*(\Delta_m^n) + s_1(\tau)$, which means that $s_1(\tau) \leq l^*(2\Delta_m^n) - l^*(\Delta_m^n)$.
\end{rmk}

\begin{rmk}
    \label{rmk_nombre_segments}
    Any triangulation $\tau$ can be refined into a maximal triangulation. Hence, we have $s_0(\tau) \leq l^*(\Delta_m^n)$ and $s_1(\tau) \leq l^*(2\Delta_m^n) - l^*(\Delta_m^n)$.
\end{rmk}

\subsubsection{The special case of dimensions $5$ and $6$}

\begin{lem}
    \label{lemma_dehn_sommerville}
    Let $n$ be $5$ or $6$. For any positive integers $m$ and $n$, for any triangulation $\tau$ of $\Delta_m^n$, one has $s_2(\tau) \customleq s_2(m)$, where $s_2(m)$ is the number of triangles in a primitive triangulation of $\Delta_m^n$ and whose relative interior is contained in the interior of $\Delta_m^n$.
\end{lem}

\begin{proof}
    A triangulation $\tau$ of $\Delta_m^n$ induces a triangulation $\widetilde{\tau^*}$ of $\widetilde{(\Delta_m^n)^*}$, which is a compact $PL$-manifold without boundary. We denote by $s_k(\widetilde{\tau^*})$ the number of $k$-simplices of this triangulation. Hence, for any $k \in \{0,...,n \}$, the Dehn-Sommerville equation (see \cite{dehnsommerville} or e.g. the expository note \cite{dehnsommerville2}) for the triangulated manifold $\widetilde{(\Delta_m^n)^*}$ is
    \begin{equation*}
        s_k(\widetilde{\tau^*}) = \sum_{i=k}^{n} (-1)^{i+n} {i+1 \choose k+1} s_i(\widetilde{\tau^*})
    \end{equation*}
    
    The same relations are asymptotically valid for the numbers $s_k(\tau)$, as the number of simplices of each dimension lying in the copies of $\partial \Delta_m^n$ is bounded above by the evaluation at $m$ of a polynomial of degree $n-1$.

    \textbf{Case n = 5}: We write $s_k^*$ instead of $s_k(\widetilde{\tau^*})$ and we denote by $s_k^p$ the number of $k$-simplices in a primitive triangulation of $\widetilde{(\Delta_m^5)^*}$. The Dehn-Sommerville relations are
    \begin{equation*}
        \left\{\begin{array}{@{}l@{}}
            s_0^* = -s_0^* + 2s_1^* - 3s_2^* + 4s_3^* - 5s_4^* + 6s_5^*\\
            s_2^* = -s_2^* + 4s_3^* - 10 s_4^* + 20 s_5^*\\
            s_4^* = -s_4^* + 6s_5^*
        \end{array}\right.
    \end{equation*}
    We can reduce this system to obtain
    \begin{equation*}
        \left\{\begin{array}{@{}l@{}}
            -2s_0^* + 2s_1^* - s_2^* + s_5^* = 0\\
            -2s_2^* + 4s_3^* - 10 s_5^* = 0\\
            -s_4^* + 3s_5^* = 0
        \end{array}\right.
    \end{equation*}

    It is clear that $s_0^* \leq s_0^p$ and $s_5^* \leq s_5^p$. Furthermore, Remark \ref{rmk_nombre_segments} gives $s_1^* \leq s_1^p$. The first equation in our system is $s_2^* = -2s_0^* + 2s_1^* + s_5^*$. If the triangulation $\tau$ is maximal, then $s_0^* = s_0^p$ and it is clear that $-2s_0^* + 2s_1^* + s_5^* \leq -2s_0^p + 2s_1^p + s_5^p = s_2^p$. If $\tau$ is not maximal, then it can be refined into a maximal triangulation by successively adding the missing vertices. With each refinement step, the number of edges grows by at least one. Hence the quantity $-2s_0^* + 2s_1^* + s_5^*$ does not decrease when we refine the triangulation. We obtain $s_2^* \leq s_2^p$, which then gives $s_2(\tau) \customleqcinq s_2(m)$.

    Additionally, the second equation gives $4s_3^* = 2s_2^* + 10 s_5^*$. As $s_2^* \leq s_2^p$ and $s_5^* \leq s_5^p$, we obtain $4s_3^* = 2s_2^* + 10 s_5^* \leq 2s_2^p + 10 s_5^p = 4s_3^p$. Moreover, the second equation also gives $2s_3^* - s_2^* = 5s_5 \leq 5s_5^p = 2s_3^p -s_2^p$. Finally, the last equation directly gives $s_4^* \leq s_4^p$.

    \textbf{Case n = 6}: The Dehn-Sommerville relations are
    \begin{equation*}
        \left\{\begin{array}{@{}l@{}}
            1 = s_0^* - s_1^* + s_2^* - s_3^* + s_4^* - s_5^* + s_6^*\\
            s_1^* = -s_1^* + 3 s_2^* - 6 s_3^* + 10 s_4^* - 15 s_5^* + 21 s_6^*\\
            s_3^* = -s_3^* + 5 s_4^* - 15 s_5^* + 35 s_5^*\\
            s_5^* = -s_5^* + 7 s_6^*
        \end{array}\right.
    \end{equation*}
    We reduce the system to obtain
    \begin{equation*}
        \left\{\begin{array}{@{}l@{}}
            20 s_0^* - 8 s_1^* + 2 s_2^* - s_6^* = 20\\
            - 4 s_1^* + 6 s_2^* - 4 s_3^* + 7 s_6^* = 0\\
            - 4 s_3^* + 10 s_4^* - 35 s_6^* = 0\\
            -2 s_5^* + 7 s_6^* = 0
        \end{array}\right.
    \end{equation*}

    The first equation is $2s_2^* = -20 - 20 s_0^* + 8 s_1^* + s_6^*$. When refining the triangulation $\widetilde{\tau^*}$ by adding a vertex, the number of edges grows by at least $6$. Indeed, the added vertex is contained in at least one $6$-simplex, which means that the refined triangulation contains $7$ new edges, one for each vertex of the $6$-simplex. If the added point originally lies on an edge of the triangulation, then this edge will not be an edge of the refined triangulation. Hence the lower bound of $6$. Thus, $-20 - 20 s_0^* + 8 s_1^* + s_6^*$ grows with each refinement step, which gives $s_2^* \leq s_2^p$, as the inequality holds if $\tau$ is maximal.

    Additionally, substituting $6s_2^*$ with $- 60 - 60 s_0^* + 24 s_1^* + 3 s_6^*$ in the second equation, we get $4 s_3^* = -30 - 30 s_0^* + 10 s_1^* + 5 s_6^*$ and using the arguments above we obtain $s_3^* \leq s_3^p$.

    Then, the third equation $10 s_4^* = 4 s_3^* + 35 s_6^*$ and the last equation $2 s_5^* = 7 s_6^*$ directly give $s_4^* \leq s_4^p$ and $s_5^* \leq s_5^p$. We can also obtain $- 4 s_3^* + 10 s_4^* \leq - 4 s_3^p + 10 s_4^p$.

\end{proof}

\begin{rmk}
    In higher dimensions, the results are already much weaker: for instance, when $n=7$ we do not obtain $s_2^* \leq s_2^p$, but only $s_2^* - s_5^* + 6 s_7^* \leq s_2^p - s_5^p + 6 s_7^p$.
\end{rmk}

\subsection{Simplex vision}

\begin{sloppypar}
    Let $P$ be a $T^2$-polynomial of degree $m$ in $n$ variables and let $O$ be an integer point that is generic with respect to $\tau(P)$. Let $\sigma^k$ be a $k$-simplex of $\tau(P)$ whose relative interior lies in the interior of $\Delta_m^n$. Let $X$ be the barycenter of $\sigma^k$. There exists a positive real number $\epsilon$ and two distinct $n$-simplices $B_O(\sigma^k)$ and $F_O(\sigma^k)$ in the star of $\sigma^k$ such that, for all $t \in (0, \epsilon)$, the point $X + t \overrightarrow{OX}$ (respectively, $X - t \overrightarrow{OX}$) lies in the interior of $B_O(\sigma^k)$ (respectively, $F_O(\sigma^k)$). The simplex $B_O(\sigma^k)$ (respectively, $F_O(\sigma^k)$) is called the simplex \textit{behind} $\sigma^k$ (respectively, \textit{in front of} $\sigma^k$) with respect to $O$.
\end{sloppypar}

\begin{lem}
    \label{lemme_indices_devant_derriere}
    We have $i^O(B_O(\sigma^k)) \geq n-k$ and $i^O(F_O(\sigma^k))\leq k+1$.
\end{lem}

\begin{proof}
    The simplex $\sigma^k$ is an intersection of $n-k$  top-dimensional faces of $B_O(\sigma^k)$ (respectively, $F_O(\sigma^k)$). If one of these $n$-faces is $O$-non-visible (respectively, $O$-visible), then we reach a contradiction, as for all $t \in (0, \epsilon)$, the point $X + t \overrightarrow{OX}$ (respectively, $X - t \overrightarrow{OX}$) cannot lie in the interior of $B_O(\sigma^k)$ (respectively, $F_O(\sigma^k)$).
\end{proof}

\begin{cor}
    If $k=0$, then $i^O(B_O(\sigma^k)) = n$ and $i^O(F_O(\sigma^k)) = 1$. \hfill $\qedsymbol$
\end{cor}

\begin{lem}
    \label{lemme_technique_signes}
    Suppose that the vertices of $\sigma^k$ all bear the same sign and that $B_O(\sigma^k)$ (respectively, $F_O(\sigma^k)$) has real critical copies. Then we have $i^O(B_O(\sigma^k)) = n-k$ (respectively, $i^O(F_O(\sigma^k)) = k+1$).
\end{lem}

\begin{proof}
    The simplex $\sigma^k$ is an intersection of $n-k$  top-dimensional $O$-visible (respectively, $O$-non-visible) faces of $B_O(\sigma^k)$ (respectively, $F_O(\sigma^k)$). Hence, the set of vertices of $\sigma^k$ contains $V_O^+(B_O(\sigma^k))$ (respectively, $V_O^-(F_O(\sigma^k))$). If $B_O(\sigma^k)$ (respectively, $F_O(\sigma^k)$) has more than 
    $n-k$
    $O$-visible (respectively, $O$-non-visible) $n$-faces, then the cardinal of $V_O^+(B_O(\sigma^k))$ (respectively, $V_O^-(F_O(\sigma^k))$) is less than $k+1$. So, at least one vertex of $\sigma^k$ belongs to $V_O^+(B_O(\sigma^k))$ (respectively, $V_O^-(F_O(\sigma^k))$) and at least one vertex of $\sigma^k$ belongs to $V_O^-(B_O(\sigma^k))$ (respectively, $V_O^+(F_O(\sigma^k))$). As these vertices bear the same sign, $B_O(\sigma^k)$ (respectively, $F_O(\sigma^k)$) has no real critical copy. We reach a contradiction. 
\end{proof}

\begin{lem}
    \label{lemme_technique_n-1_simplexes}
    If $n>2$ and $k = n-1$, then
    \begin{itemize}
        \item either $i^O(B_O(\sigma^k)) \neq 1$,
        \item or $i^O(F_O(\sigma^k)) \neq n$,
        \item or at least one of the simplices $B_O(\sigma^k)$ and $F_O(\sigma^k)$ is not the dilation by $2$ of a primitive simplex.
    \end{itemize}
\end{lem}

\begin{proof}
    Suppose that $B_O(\sigma^k)$ and $F_O(\sigma^k)$ are both dilations by $2$ of primitive simplices, that $i^O(B_O(\sigma^k)) = 1$ and that $i^O(F_O(\sigma^k)) = n$. Let $V_1$ (respectively, $V_2$) be the vertex of $B_O(\sigma^k)$ (respectively, $F_O(\sigma^k)$) that does not belong to $F_O(\sigma^k)$ (respectively, $B_O(\sigma^k)$). As the simplices are dilations by $2$ of primitive simplices, the middle $M$ of the segment $[V_1, V_2]$ is an integer point and lies on the hyperplane that contains $B_O(\sigma^k) \bigcap F_O(\sigma^k)$. Furthermore, $B_O(\sigma^k) \bigcap F_O(\sigma^k)$ is the dilation by $2$ of a primitive $(n-1)$-simplex. Then, as $i^O(B_O(\sigma^k)) = 1$ and $i^O(F_O(\sigma^k)) = n$, the point $M$ lies in the relative interior of $B_O(\sigma^k) \bigcap F_O(\sigma^k)$. We obtain a contradiction : the dilation by $2$ of a primitive $(n-1)$-simplex does not contain any integer point in its relative interior. 
\end{proof}

\begin{cor}
    Suppose that $k = n-1$ and $B_O(\sigma^k)$ and $F_O(\sigma^k)$ are both dilations by $2$ of primitive simplices. If all the vertices of $\sigma^k$ bear the same sign, then at least one of $B_O(\sigma^k)$ and $F_O(\sigma^k)$ does not have any real critical copy.
\end{cor}

\begin{proof}
    It suffices to combine the last two lemmas.
\end{proof}

\subsection{Number of $n$-simplices of a given index}

\label{section_number_of_simplices_given_index}

\begin{prop}
    \label{prop_borne_indice_general}
    Let $P$ be a $T$-polynomial with Newton polytope $\Delta_m^n$ and let $O$ be an integer point that is generic with respect to $P$. For any $k \in \{0,...,n \}$, we have that
    \begin{center}
    $s_k(\tau(P)) \customeq \sum_{i=0}^k { n-i \choose n-k } \mathcal{S}_{n-i}^O(\tau(P)) \customeq \sum_{i=0}^k { n-i \choose n-k } \mathcal{S}_{i+1}^O(\tau(P))$.
    \end{center}
\end{prop}

\begin{proof}
    Let $k \in  \{0,...,n \}$. By Lemma \ref{lemme_indices_devant_derriere}, for any $k$-simplex $\sigma^k$ of $\tau(P)$ lying in the interior of $\Delta_m^n$, we have $i^O(B_O(\sigma^k))\geq n-k$ and $i^O(F_O(\sigma^k)) \leq k+1$.
    Let $i \in \{ 0,...,k \}$.
    Let $\sigma^n_1$ (respectively, $\sigma^n_2$) be an $n$-simplex of index $n-i$ (respectively, of index $i+1$). It has $n-i$ $O$-visible (respectively, $O$-non-visible) $(n-1)$-faces and hence ${ n-i \choose n-k }$ distinct $k$-faces that are intersections of $O$-visible (respectively, $O$-non-visible) $(n-1)$-faces. We can suppose that all these $k$-faces are contained in the interior of $\Delta_m^n$, as the number of $k$-simplices of $\tau(P)$ that are contained in $\partial \Delta_m^n$ is asymptotically negligible. For any simplex $\sigma^k$ that belongs to the set of these $k$-faces, $B_O(\sigma^k) = \sigma^n_1$ (respectively, $F_O(\sigma^k) = \sigma^n_2$). The result follows.
\end{proof}

\begin{prop}
    Let $P$ be a $T$-polynomial with Newton polytope $\Delta_m^n$ and let $O$ be an integer point that is generic with respect to $\tau(P)$. For any $k \in \{0,...,n \}$, we have that
    \begin{center}
    $\mathcal{S}_{n-k}^O(\tau(P)) \customeq \sum_{i = 0}^k (-1)^{k-i} s_i(\tau(P)) \frac{1}{(k-i)!}\Pi_{j=1}^{k-i}(n-k+j) \customeq \mathcal{S}_{k+1}^O(\tau(P))$.
    \end{center}
\end{prop}

\begin{proof}
    If $k = 0$, Proposition \ref{prop_borne_indice_general} gives $s_0(\tau(P)) \customeq \mathcal{S}_{n}^O(\tau(P))$. Suppose that for any $l \leq k-1$, we have $\mathcal{S}_{n-l}(\tau(P)) = \sum_{i = 0}^l (-1)^{l-i} s_i(\tau(P)) \frac{1}{(l-i)!}\Pi_{j=1}^{l-i}(n-l+j)$. By Proposition \ref{prop_borne_indice_general}, we have $s_k(\tau(P)) \customeq \sum_{i=0}^k { n-i \choose n-k } \mathcal{S}_{n-i}^O(\tau(P))$, which gives 
    \begin{align*}
        {S}_{n-k}^O(P) & \customeq  s_k(\tau(P)) - \sum_{i=0}^{k-1} { n-i \choose n-k } \mathcal{S}_{n-i}^O(\tau(P)) & \\ 
        & =  s_k(\tau(P)) - \sum_{i=0}^{k-1} { n-i \choose n-k } \sum_{j = 0}^i (-1)^{i-j} s_j(\tau(P)) \frac{1}{(i-j)!}\Pi_{l=1}^{i-j}(n-i+l).    
    \end{align*}
    Let $q \in \{0,...,k-1 \}$. 
    In the last expression, the sum of the terms that contain $s_q(\tau(P))$ is 
    \begin{align*}
        -\sum_{i=q}^{k-1} {n-i \choose n-k} (-1)^{i-q} s_q(\tau(P)) \frac{1}{(i-q)!}\Pi_{l=1}^{i-q}(n-i+l) & =  (-1)^{k-q} \frac{(n-q)!}{(n-k)!(k-q)!} s_q(\tau(P))\\
        &= (-1)^{k-q} \frac{(n-k+1)...(n-q)}{(k-q)!}s_q(\tau(P)).
    \end{align*}
    It follows that $\mathcal{S}_{n-k}(\tau(P)) \customeq \sum_{i = 0}^k (-1)^{k-i} s_i(\tau(P)) \frac{1}{(k-i)!}\Pi_{j=1}^{k-i}(n-k+j)$. The result involving $\mathcal{S}_{k+1}(\tau(P))$ is obtained in a similar way.
\end{proof}


\section{Bounds on the Betti numbers of $T^2$-hypersurfaces}

\label{section_proof_bounds}

\subsection{Bounds on some sums of Betti numbers}

\begin{sloppypar}
    Let $P$ be a $T^2$-polynomial of degree $m$ in $n$ variables and let $A$ be the $T^2$-hypersurface in $\P^n$ defined by $P$.
\end{sloppypar}



\begin{thm}
    \label{thm_partial_sums_betti_numbers}
    For any positive integer $k$, one has
    \begin{center}
    $\sum_{i=0}^{k-1} \frac{(k+n-1-i)!}{(k-1-i)!n!} b_i(\R A) \customleq 
    \sum_{i=0}^{k-1} \frac{(k+n-1-i)!}{(k-1-i)!n!} h^{i,n-1-i}(\C A)$.\end{center}
    
    Additionally, one has $b_1(\R A) \customleq h^{1,n-2}(\C A)$.

    Moreover, if $n \leq 6$, then for any non-negative integer $k\leq n$, one has
    \begin{center}$\sum_{i=0}^k { n-i \choose n-k } b_i(\R A) \customleq \sum_{i=0}^k { n-i \choose n-k } h^{i,n-i}(\C A)$.\end{center}
\end{thm}

\begin{proof}
    First, by Corollary \ref{corollary_number_simplices_number_points}, one has
    \begin{equation*}
        l^*(k\Delta_m^n) \customeq 2^nl^*(k\Delta_{\frac{m}{2}}^n) \geq 2^n\sum_{i=0}^{k-1} {k-1 \choose i} s_i(\tau(P)).
    \end{equation*}
    Let $O$ be an integer point that is generic with respect to $\tau(P)$.
    Proposition \ref{prop_borne_indice_general} states that
    \begin{equation*}
        s_i(\tau) \customeq \sum_{j=0}^i { n-j \choose n-i } \mathcal{S}_{n-j}^O(\tau(P))
    \end{equation*}
    Hence, one has 
    \begin{equation*}
        l^*(k\Delta_m^n) \customgeq 2^n\sum_{i=0}^{k-1} \left( {k-1 \choose i} \sum_{j=0}^i { n-j \choose n-i } \mathcal{S}_{n-j}^O(\tau(P)) \right) = \sum_{i=0}^{k-1} \frac{(k+n-1-i)!}{(k-1-i)!n!} 2^n\mathcal{S}_{n-i}^O(\tau(P)).
    \end{equation*}
    It remains to apply Theorem \ref{thmitenbergshustin} (\cite{itenbergshustincritpts}) and the Morse inequalities to obtain
    \begin{equation*}
        \sum_{i=0}^{k-1} \frac{(k+n-1-i)!}{(k-1-i)!n!} b_i(\R A) \customleq l^*(k\Delta_m^n).
    \end{equation*}
    The equality $l^*(k\Delta_m^n) \customeq \sum_{i=0}^{k-1} \frac{(k+n-1-i)!}{(k-1-i)!n!} h^{i,n-1-i}(\C A)$ is a direct application of the work of V. I. Danilov and A. G. Khovanskii in \cite{danilkhovan}.
    
    For the proof of the last statement, a similar computation can be done, but one has to use Lemma \ref{lemma_dehn_sommerville} instead of Corollary \ref{corollary_number_simplices_number_points}.

    To obtain the statement concerning the first Betti number, it suffices to prove that \begin{center}$\mathcal{S}_{n-1}^O(\tau(P)) \customleq l^*(2\Delta_{\frac{m}{2}}^n) - (n+1) l^*(\Delta_{\frac{m}{2}}^n)$.\end{center}
    Indeed, the equalities $2^nl^*(2\Delta_{\frac{m}{2}}^n) - 2^n(n+1)l^*(\Delta_{\frac{m}{2}}^n) \customeq l^*(2\Delta_{m}^n) - (n+1) l^*(\Delta_{m}^n)$ and $h^{1,n-2}(\C A) \customeq l^*(2\Delta_{m}^n) - (n+1) l^*(\Delta_{m}^n)$, the second one being a direct application of \cite{danilkhovan}, will then finish the proof.
    
    Let $v_1$,...$v_r$ be the integer points lying in the interior of $\Delta_{\frac{m}{2}}^n$ and that are not vertices of $\frac{1}{2}\tau(P)$. Consider the sequence of triangulations $\tau_0$, ..., $\tau_r$ such that 
    \begin{itemize}
        \item $\tau_0$ is $\tau$;
        \item for any integer $i\in \{1,...,r\}$, the triangulation $\tau_i$ is obtained by subdividing each $n$-dimensional simplex $\sigma$ of $\tau_{i-1}$ that contains $2v_i$ by taking the cone with vertex $2v_i$ over each $(n-1)$-face of $\sigma$ that does not contain $2v_i$.
    \end{itemize}
    We have $s_0(\tau_r) \customeq l^*(\Delta_{\frac{m}{2}}^n)$. Moreover, for any integer $i \in \{1,...,r \}$, one has $s_1(\tau_{i}) \geq s_1(\tau_{i-1}) + n$.
    Indeed, for any integer $i \in \{1,...,r \}$, consider a top-dimensional simplex $\sigma$ of $\tau_{i-1}$ containing $2v_i$. The triangulation $\tau_i$ contains all the edges whose extremities are $2v_i$ and a vertex of $\sigma$. These $n+1$ edges are not in $\tau_{i-1}$. Furthermore, $\tau_i$ contains all of the edges in $\tau_{i-1}$ if $2v_i$ does not lie on an edge of $\tau_{i-1}$. If $2v_i$ lies on an edge of $\tau_{i-1}$, then $\tau_i$ contains all of the edges of $\tau_{i-1}$ except one.
    
    Hence, the quantity $\left(s_1(\tau_i)-ns_0(\tau_i)\right) - \left(s_1(\tau_{i-1})-ns_0(\tau_{i-1})\right)$ is non-negative. This implies that $\left(s_1(\tau_r)-ns_0(\tau_r)\right) - \left(s_1(\tau(P))-ns_0(\tau(P))\right) \geq 0$.

    From Remark \ref{rmk_nombre_segments}, we know that $s_1(\tau_r) \leq l^*(2\Delta_{\frac{m}{2}}^n) - l^*(\Delta_{\frac{m}{2}}^n)$.
    Using Proposition \ref{prop_borne_indice_general}, one can write $\mathcal{S}_{n-1}^O(\tau(P)) + n\mathcal{S}_{n}^O(\tau(P)) \customeq s_1(\tau(P))$,
    which implies that
    \begin{align*}
        \mathcal{S}_{n-1}^O(\tau(P)) & \customeq s_1(\tau(P)) - n\mathcal{S}_{n}^O(\tau(P))\\ 
        &\customeq s_1(\tau(P)) - ns_0(\tau(P)) \\
        &\customleq s_1(\tau_r) - ns_0(\tau_r) \\
        &\customleq l^*(2\Delta_{\frac{m}{2}}^n) - (n+1) l^*(\Delta_{\frac{m}{2}}^n).
    \end{align*}
\end{proof}



\subsection{Non-existence of asymptotically maximal families}

\subsubsection{A bound when $V(\tau(P))$ is bounded from above}

\begin{sloppypar}
    Let $0 \leq a \leq b \leq 1$. Let $Q,R \in \mathbb{R}[m]$ be polynomials of degree at most $n-1$.
\end{sloppypar}

\begin{prop}
    \label{thm_nombre_de_betti_total_cas_1}
    Let $P \in TP^n(a, Q, b, R)$ be a polynomial of degree $m$ and let $O$ be an integer point that is generic with respect to $\tau(P)$. We have $\sum_{i=1}^n \mathcal{S}_i^O(\tau(P)) \customleq \frac{m^n}{2^n} - 2\frac{1}{n!2^n}(1-b)m^n$. 
\end{prop}

\begin{proof}
    Let $D$ be the number of integer points with even coordinates in the interior of $\Delta_m^n$ that are not vertices of $\tau(P)$. We have that $D \customgeq l^*(\Delta_{\frac{m}{2}}^n) - b\times \frac{1}{n!2^n}m^n$. We refine $\tau(P)$ by successively adding each of these points as a new vertex. Every time a subdivision occurs, either the considered new vertex lies in the interior of an $n$-simplex of the triangulation and the number of $n$-simplices grows by $n$, or the considered vertex is in the relative interior of a smaller dimensional face of the triangulation and the number of $n$-simplices in the triangulation grows by at least $2$, as there are at least two $n$-simplices adjacent to the face and each of them will be divided in at least two $n$-simplices. Hence the number of $n$-simplices of $\tau(P)$ satisfies $\#\tau(P) \customleq \frac{m^n}{2^n} - 2(l^*(\Delta_{\frac{m}{2}}^n) - b\times \frac{1}{n!2^n}m^n)$. This finishes the proof.
\end{proof}

\begin{cor}
    Let $P \in TP^n(a, Q, b, R)$ be a polynomial of degree $m$ and let $A$ be the $T^2$-hypersurface in $\P^n$ defined by $P$. One has
    \begin{equation*}b_*(\mathbb{R}A) \customleq m^n - \frac{2}{n!}m^n(1-b) \customeq b_*(\C A) - \frac{2}{n!}m^n(1-b).\end{equation*}
\end{cor}

\begin{proof}
    The result follows from the inequalities $b_*(\mathbb{R}A) \customleq \sum_{i=0}^{n-1}c_i^-(P) +c_{n-i}^+(P) \leq \sum_{i=1}^n 2^n \mathcal{S}_i^O(\tau(P))$ and Proposition \ref{thm_nombre_de_betti_total_cas_1}.
\end{proof}

\subsubsection{A bound when $V(\tau(P))$ is bounded from below}

\begin{sloppypar}
    In this section, we suppose that $n>2$. Let $P \in TP^n(a,Q,b,R)$ be a $T^2$-polynomial of degree $m$ and let $A$ be the $T^2$-hypersurface defined by $P$. Let $O$ be an integer point that is generic with respect to $\tau(P)$. Let $V^0$ be a vertex of $\tau(P)$ lying in the interior of $\Delta_m^n$.
    We want to prove that the star of $V^0$ contains either at least one $n$-simplex with no real critical copy or at least one $n$-simplex that is not the dilation by $2$ of a primitive simplex. To fix the notations, we suppose that $V^0$ bears the sign $-$.
\end{sloppypar}

\begin{lem}
    \label{lemme_preuve_asympto_simplexes_n-1_n}
    Let $k = n-1, n$. Suppose that $V^0$ is a vertex of a $k$-simplex $V^k$ whose vertices all bear the sign $-$. Then
    \begin{itemize}
        \item either the star of $V^k$ contains an $n$-simplex with no real critical copy,
        \item or the star of $V^k$ contains an $n$-simplex that is not the dilation by $2$ of a primitive simplex.
    \end{itemize}
\end{lem}

\begin{proof}
    If $k = n$, then $V^k$ is an $n$-simplex with no real critical copy.
    If $k = n-1$, then Lemmas \ref{lemme_technique_signes} and \ref{lemme_technique_n-1_simplexes} imply the desired result.
\end{proof}

\begin{lem}
    \label{lemme_preuve_asympto_simplexes_0_n-2}
    Let $k$ be an integer such that $0 \leq k \leq n-2$.
    Suppose that $V^0$ is a vertex of a $k$-simplex $V^k$ whose vertices all bear the sign $-$. Then
    \begin{itemize}
        \item either the star of $V^k$ contains an $n$-simplex with no real critical copy,
        \item or $V^0$ is the vertex of a $(k+1)$-simplex whose vertices all bear the sign $-$.
    \end{itemize}
\end{lem}

\begin{proof}
    If one of the simplices $B_O(V^k)$ and $F_O(V^k)$ does not have any real critical copy, then we are done. Otherwise, by Lemma \ref{lemme_technique_signes}, we have $i^O(B_O(V^k)) = n-k$ and $i^O(F_O(V^k)) = k+1$. Furthermore, the vertices of $B_O(V^k)$ (respectively, $F_O(V^k)$) that are not vertices of $V^k$ bear the sign $+$. Let $\sigma$ be an $n$-simplex of $\tau(P)$ containing $V^k$, adjacent to $F_O(V^k)$ and distinct from $B_O(V^k)$. If the vertex of $\sigma$ that does not belong to $F_O(V^k)$ bears the sign $+$, then
    \begin{itemize}
        \item either $\sigma$ has no real critical copies,
        \item or
        \begin{itemize}
            \item either $i^O(\sigma) = k+1$ and $V_+^O(\sigma)$ is the set of vertices of $V^k$,
            \item or $i^O(\sigma) = n-k$  and $V_-^O(\sigma)$ is the set of vertices of $V^k$.
        \end{itemize}
    \end{itemize}
    In the last two cases, $V^k$ is the intersection of all $O$-visible or $O$-non-visible $n$-faces of $\sigma$, which means that $\sigma = F_O(V^k)$ or $\sigma = B_O(V^k)$, which is impossible by hypothesis. If the vertex of $\sigma$ that does not belong to $F_O(V^k)$ bears the sign $-$, then $V^0$ is the vertex of a $(k+1)$-simplex of $\tau$ whose vertices all bear the sign $-$.
\end{proof}

\begin{prop}
    \label{prop_preuve_asympto_mauvais_simplexe_dans_chaque_etoile}
    The star of $V^0$ contains either at least one $n$-simplex with no real critical copies or at least one $n$-simplex that is not the dilation by $2$ of a primitive simplex.
\end{prop}

\begin{proof}
    If $V^0$ is a vertex of a $k$-simplex $V^k$ whose vertices all bear the sign $-$, then by Lemma \ref{lemme_preuve_asympto_simplexes_0_n-2}, either the star of $V^k$ contains an $n$-simplex with no real critical copy, or $V^0$ is the vertex of a $(k+1)$-simplex whose vertices all bear the sign $-$. We apply this result repeatedly, starting with $V^k = V^0$. We stop either if $k \leq n-2$ and an $n$-simplex with no real critical copy in $St(V^k) \subset St(V^0)$ was found, or if $k = n-1$. In the latter case, we conclude using Lemma \ref{lemme_preuve_asympto_simplexes_n-1_n}.
\end{proof}

\begin{thm}
    \label{thm_nombre_betti_total_cas2}
    If $n>2$, we have $\sum_{i=0}^{n-1}b_i(\mathbb{R}A) \customleq m^n -\frac{a}{(n+1)!} m^n$.
\end{thm}

\begin{proof}
    Let $D_1$ be the number of $n$-simplices of $\tau(P)$ that are not dilations by $2$ of a primitive simplex. Notice that each such simplex has Euclidean volume greater than $\frac{2^{n}}{n!}$. Hence the number of $n$-simplices of $\tau(P)$ is less than or equal to $(\frac{m}{2})^n - D_1$. Let $D_2$ be the number of simplices of $\tau(P)$ that have no real critical copies. We have 
    \begin{align*}\sum_{i=0}^{n-1}b_i(\mathbb{R}A) &\customleq \sum_{i=0}^{n-1} c_i-(P) + c_{n-i}^+(P) \\ 
    &\leq \sum_{i=1}^{n} 2^n \bar{\mathcal{S}}_i^O(\tau(P)) \\
    &\leq 2^n ((\frac{m}{2})^n - D_1 - D_2).\end{align*}
    By Proposition \ref{prop_preuve_asympto_mauvais_simplexe_dans_chaque_etoile}, the star of each vertex of $\tau(P)$ that lies in the interior of $\Delta_m^n$ contains either at least one $n$-simplex with no real critical copies or at least one $n$-simplex $S$ that is not the dilation by $2$ of a primitive simplex. Each such simplex has $n+1$ vertices. Hence, $D_1 + D_2 \customgeq \frac{1}{n+1}\frac{a}{n!2^n} m^n$.
    We obtain \begin{center}$\sum_{i=0}^{n-1}b_i(\mathbb{R}A) \customleq \sum_{i=0}^{n-1} c_i^-(P) + c_{n-i}^+(P) \customleq m^n - \frac{a}{(n+1)!} m^n$.\end{center}
\end{proof}

\subsubsection{Proof of non-existence of asymptotically maximal families of hypersurfaces}

\begin{defin}
    Let $\Lambda$ be an infinite subset of $\mathbb{N}$.
    Let $\mathcal{F} = (A_m)_{m \in \Lambda}$ be a family of real algebraic hypersurfaces in $\P^n$, where $A_m$ is defined by a polynomial with Newton polytope $\Delta_m^n$. The family $\mathcal{F}$ is said to be \textit{asymptotically maximal} if $b_*(\mathbb{R}A_m) \customeq m^n$.
\end{defin}

\begin{thm}
    \label{main_theorem}
    Let $n >2$ and let $\Lambda$ be an infinite subset of $2\mathbb{N}$.
    Let $\mathcal{F} = (A_m)_{m \in \Lambda}$ be a family of $T^2$-hypersurfaces in $\P^n$ indexed by their degree. The family $\mathcal{F}$ is not asymptotically maximal.
\end{thm}

\begin{proof}
    Let $a\in (0,1)$ and let $Q \in \mathbb{R}[m]$ be a polynomial of degree at most $n-1$. There exists a polynomial $R \in \mathbb{R}[m]$ of degree at most $n-1$ such that $TP^n = TP^n(0,0,a,Q) \bigcup TP^n(a,Q,1,R)$. Hence, at least one of the sets $\mathcal{F} \bigcap TP^n(0,0,a,Q)$ and $\mathcal{F} \bigcap TP^n(a,Q,1,R)$ has infinite cardinal.
    This means that $\mathcal{F}$ admits a subfamily $\mathcal{F}'$ to which either Proposition \ref{thm_nombre_de_betti_total_cas_1} or Theorem \ref{thm_nombre_betti_total_cas2} can be applied to prove that $\mathcal{F}'$ is not asymptotically maximal. Hence, $\mathcal{F}$ is not asymptotically maximal.
\end{proof}

\subsection{Comparing $SQM$-hypersurfaces to $T^2$-hypersurfaces}

\begin{sloppypar}
    It turns out that Theorems \ref{thm_partial_sums_betti_numbers} and \ref{main_theorem} do not apply to arbitrary $SQM$-hypersurfaces. It is a consequence of the following theorem.
\end{sloppypar}

\begin{thm}
    \textbf{(\cite{thesearnal}, Theorem 6.1.1)} For any $n \geq 3$ and any $i = 0,...,n-1$, there exists a real number $a_i^n > 0$ and an asymptotically maximal family $(A_m^n)_{m \in \N}$ of real algebraic hypersurfaces in $\P^n$ indexed by their degree such that $b_i(\R A_m^n) \customgeq h^{i,n-i}(\C A_m^n) + a_i^n m^n$.
\end{thm}

\begin{cor}
    \label{cor_arnal}
    For any $n \geq 3$ and any $i = 0,...,n-1$, there exists $a_i^n > 0$ and an asymptotically maximal family $(B_m^n)_{m \in 2\N}$ of $SQM$-hypersurfaces indexed by their degree such that $b_i(\R B_m^n) \customgeq h^{i,n-i}(\C A_m^n) + a_i^n m^n$.
\end{cor}

\begin{proof}
    The real parts of the hypersurfaces $(A_d^n)_{d \in \N}$ can be taken to be compact affine. By applying a suitable translation, one can obtain hypersurfaces whose real zeroes are contained in $(\R_{>0})^n$. It remains to multiply by $2$ all the exponents in the polynomials defining these hypersurfaces. The resulting polynomials define the hypersurfaces $(B_m^n)_{m \in 2 \N}$.
\end{proof}




\section{Sharpness of the bounds}

\label{section_constructions}

\begin{sloppypar}
    The bounds we have obtained can be thought of as bounds on the number of critical points of $T^2$-polynomials or as bounds on the Betti numbers of $T^2$-hypersurfaces. The two cases raise different results regarding the sharpness of the bounds. We mainly focus on the number of extrema and the number of connected components, as these seem to be the most accessible cases.
\end{sloppypar}

\subsection{Number of extrema}

\begin{prop}
    For any integer $n$, there exists a family $(P_m^n)_{m \in 2\N}$ of $T^2$-polynomials in $n$ variables indexed by their degree and a family of points $(O_m^n)_{m \in 2\N}$ in $\Z^n$ such that $\Bar{\mathcal{S}}_n^{O_m^n}(\tau(P_m^n)) \customeq \frac{1}{2^n n!}m^n$.
\end{prop}

\begin{proof}
    Let $n$ be a positive integer and let $m$ be an even positive integer. We describe a convex triangulation $\tau$ of $\Delta_m^n$ with signs on its vertices.
    \begin{itemize}
        \item First, subdivide $\Delta_{\frac{m}{2}}^n$ using all the hyperplanes that are parallel to an $(n-1)$-face of $\Delta_{\frac{m}{2}}^n$ and that contain integer points. This subdivision is clearly convex.
        \item Refine the subdivision into a convex triangulation. Such a triangulation is maximal.
        \item Every integer point belonging to $\Delta_{\frac{m}{2}}^n$ receives the sign $+$ if the sum of its coordinates is even and the sign $-$ otherwise.
        \item Dilate the triangulation by $2$ to obtain a triangulation of $\Delta_m^n$.
    \end{itemize}
    Choose an integer point $O_m^n$ that is generic with respect to $\tau$ and that is contained in the intersection of cones $C := \bigcap_{v \in \Z^n \cap \mathring\Delta_m^n}\{ v- \sum_{i=1}^n t_ie_i \; | \; \forall i \in \{1,...,n\}, t_i \in \R_{>0} \}$.
    
    Let $V = (x_0,...,x_n)$ be an integer point with even coordinates lying in the interior of $\Delta_m^n$. The triangulation $\tau$ contains the $n$-simplex $\sigma_V$ with vertices $V$, $V + (2,0,...,0)$, $V + (0,2,0,..,0)$,..., $V + (0,...,0,2)$. The cone with vertex $O_m^n$ over all the $n-1$-faces of $\sigma_V$ that contain $V$ does not intersect the interior of $\sigma_V$, hence the $O_m^n$-index of $\sigma_V$ is $n$ and the $O_m^n$-root of $\sigma_V$ is $V$. Furthermore, the sign of $V$ is opposite to the sign of the other vertices of $\sigma_V$. Hence $\sigma_V$ has $2^n$ real $O_m^n$-critical copies. The result follows.
\end{proof}

\begin{rmk}
    For some fixed positive integers $n$ and $m$, the number of connected components of the real part of the real algebraic variety defined by $P_m^n$ is small: the real part is homeomorphic to $\frac{m}{2}$ spheres of dimension $n-1$.
\end{rmk}

\begin{rmk}
    Similarly, one also has $\Bar{\mathcal{S}}_1^{O_m^n}(\tau(P_m^n)) \customeq \frac{1}{2^n n!}m^n$.
\end{rmk}

\subsection{A better bound in the case of curves}

\begin{sloppypar}
    The sharpness proved in the previous section for extrema does not imply that the bound on the number of connected components from Theorem \ref{thm_partial_sums_betti_numbers} is sharp. In fact, in the case of curves, one can find a better bound, expressed in function of the Euclidean area of the Newton polygon of the considered curves.
\end{sloppypar}

\begin{prop}
    \label{prop_better_bound_curves}
    Let $C$ be a $T^2$-curve with Newton polygon $\Delta \subset (\R_{\geq 0})^2$ and constructed using a triangulation $\tau$. Then the number of connected components of $\R C$ is less than or equal to $\frac{2Area(\Delta)}{3}$.
\end{prop}

\begin{proof}
    First, associate to every triangle of $\tau$ the number $0$. Let $B$ be a connected component of the $PL$-curve obtained during the patchworking process, that is homeomorphic to $\R C$ and that intersects the symmetric copy of $\Delta$ that lies in $(\R_{\geq0})^2$.

    If $B$ is contained in the interior of the symmetric copy of $\Delta$ that lies in $(\R_{\geq 0})^2$, then it has $4$ symmetric copies and intersects at least $3$ triangles of $\tau$. To each of these $3$ triangles, associate the number $\frac{4}{3}$ instead of $0$.

    Otherwise, if $B$ is not contained in the interior of the symmetric copy of $\Delta$ that lies in $(\R_{\geq0})^2$, then it has at most $2$ symmetric copies. If $B$ has $2$ symmetric copies, then it intersects at least $2$ triangles of $\tau$. To each of these $2$ triangles, associate the number $1$ instead of $0$. If $B$ has only $1$ symmetric copy, then it intersects at least $1$ triangle of $\tau$. To this triangle, associate the number $1$ instead of $0$.

    Doing this for every connected component $B$ of the $PL$-curve obtained during the patchworking process such that $B$ intersects the symmetric copy of $\Delta$ that lies in $(\R_{\geq 0})^2$, we associate to every triangle $T$ of $\tau$ a non-negative number $n_T$ that is not greater than $\frac{4}{3}$ and such that $\sum_{T \in \tau} n_T = b_0(\R C)$. Clearly, $\sum_{T \in \tau} n_T \leq \frac{4}{3} \#\tau \leq \frac{4}{3} \frac{Area(\Delta)}{2}$.
\end{proof}

\begin{rmk}
    Let $C$ be a $T^2$-curve of degree $d$ in $\P^2$. Proposition \ref{prop_better_bound_curves} gives $b_0(\R C) \leq \frac{d^2}{3}$, while Theorem \ref{thm_partial_sums_betti_numbers} only gives $b_0(\R C) \customleqdeux \frac{d^2}{2}$.
\end{rmk}

\begin{sloppypar}
    It turns out that the bound from Proposition \ref{prop_better_bound_curves} is sharp in certain toric surfaces.
\end{sloppypar}

\begin{prop}
    \label{prop_b0_curves}
    Let $\Pi$ be the polygon with vertices $(0,0)$, $(2,1)$, $(1,2)$.
    For any even positive integer $m$, there exists a $T^2$-curve $C_m$  whose Newton polygon is $m\Delta$ such that $b_0(\R C_m) = \frac{2Area(m\Delta)}{3}$.
\end{prop}

\begin{proof}
    For any even positive integer $m$, we describe a triangulation of $m\Pi$ and a distribution of signs at its vertices.
    \begin{itemize}
        \item Take the triangulation of $\frac{m}{2} \Pi$ defined by the union of the lines with equation $x+y = k$, $-2x+y = -k$ and $-2y+x = -k$, for all non-negative integer $k$ divisible by $3$ and smaller than or equal to $3\frac{m}{2}$.
        \item Any vertex of this triangulation receives the sign $+$.
        \item Refine the triangulation into a primitive one in the only possible way.
        \item The vertices of this triangulation to which no sign was assigned receive the sign $-$.
        \item Dilate the triangulation by $2$ to obtain a triangulation of $m\Pi$.
    \end{itemize}
    Let $V$ be a vertex of the triangulation bearing the sign $-$. It cannot lie on the boundary of $\Pi$. The star of $V$ consists of exactly three triangles. The vertices of these triangles that are different from $V$ bear the sign $+$. Hence, there are four connected components of the $PL$-curve obtained by using combinatorial patchworking on the described datum that are contained in the symmetric copies of the star of $V$. Furthermore, every triangle in the triangulation has one of its vertices that bear the sign $-$. The result follows.
\end{proof}

\begin{rmk}
    Using a variation of this construction, one can prove that the bound of Proposition \ref{prop_better_bound_curves} is asymptotically sharp in $\P^2$.
\end{rmk}

\begin{figure}[H]
    \centering
    \includegraphics[scale=0.5]{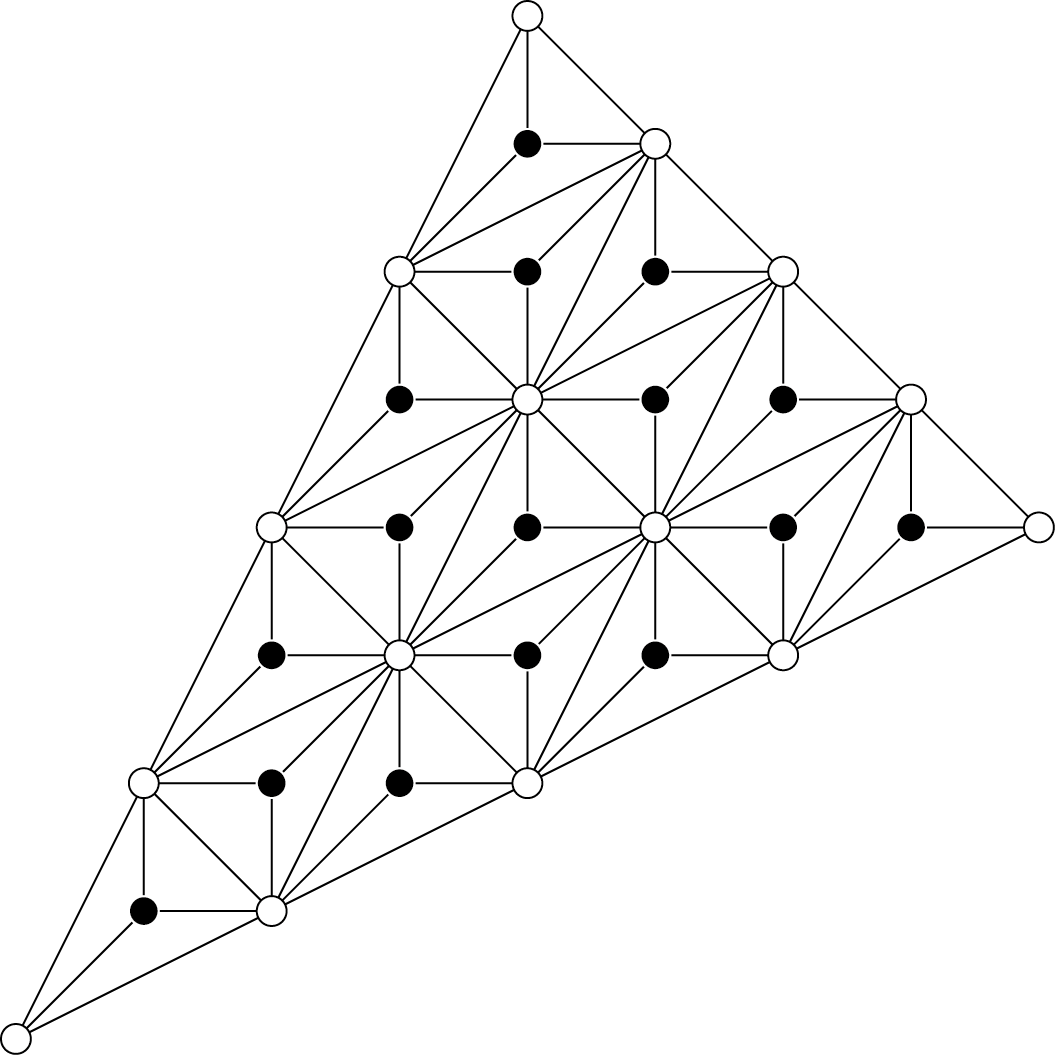}
    \caption{The described triangulation and sign distribution for $\frac{m}{2}=4$. The color of the vertices represent their sign.}
    \label{figure-harnackset}
\end{figure}

\begin{rmk}
    It is interesting to notice that triangulations similar to the one we used to prove Proposition \ref{prop_b0_curves} appear in different contexts, for example in constructions of real algebraic curves with large finite number of real points (see \cite{finite_number_real_points}) or in constructions of polynomials in two variables with prescribed numbers of real critical points of given index (see \cite{shustin_early_critpts}).
\end{rmk}

\subsection{Number of connected components}

    We generalize the construction from the previous section to all dimensions.

\begin{thm}
    \label{prop_large_b0_alldim}
    Let $n$ be a positive integer.
    Let $\Pi^n \subset \R^n$ be the polytope with vertices $(0,...,0)$, $(2,1,...,1)$, $(1,2,1,...,1)$, ..., $(1,...,1,2)$. There exists a family $(A_m^n)_{m \in 2\N}$ of $T^2$-hypersurfaces such that $b_0(\R A_m^n) \customeq \frac{n}{n+1}h^{n-1,0}(\C A_m^n)$ and such that for any even positive integer $m$, the hypersurface $A_m$ has Newton polytope $m\Pi^n$.
\end{thm}

    Let $n$ be a positive integer and let $m$ be an even positive integer. We denote by $x_1,...,x_n$ the standard coordinates in $\R^n$. Consider the subdivision of $\frac{m}{2}\Pi^n$ obtained by using the hyperplanes with equations $\sum_{i=1}^n x_i = (n+1)k$ and $-nx_j \sum_{i\in \{1,...,n \}\setminus \{ j \}} x_i = -(n+1)k$ for every integer $j \in \{1,...,n \}$ and for every non-negative integer $k$. The obtained subdivision is convex and is denoted by $\overline{\tau_m^n}$.

\begin{lem}
    \label{lemma_points_in_cells}
    Every integer point lying in the interior of $\frac{m}{2}\Pi^n$ but not on a hyperplane with equation $\sum_{i=1}^n x_i = (n+1)k$ for a any integer $k$ lies in the interior of an $n$-dimensional cell of $\overline{\tau_m^n}$. Furthermore, the interior of every $n$-dimensional cell of $\overline{\tau_m^n}$ contains at most one integer point. 
\end{lem}

\begin{proof}
    Let $p = (p_1,..., p_n)$ be an integer point lying in the interior of $\frac{m}{2}\Pi^n$ but not on a hyperplane with equation $\sum_{i=1}^n x_i = (n+1)k$ for any integer $k$. Let $j$ be the integer in $\{1,...,n \}$ such that there exists a non-negative integer $k_0$ such that the point $p' := p - j(1,...,1)$ is on the hyperplane with equation $\sum_{i=1}^n x_i = (n+1)k_0$. Let $(k_1,...,k_n)$ be the unique solution of the following system of equations.
    \begin{equation}
    \begin{cases}
      2x_1 + \sum_{i \in \{2,...,n \}}x_i =  p_1 - j\\
      x_1 + 2x_2 + \sum_{i \in \{3,...,n \}}x_i = p_2 - j\\
      ...\\
      \sum_{i \in \{1,...,n-1\}}x_i + 2x_n = p_n-j
    \end{cases}
    \end{equation}
    The numbers $k_1,...,k_n$ are integers: the determinant of the matrix associated to the system is $n+1$ and the inverse of the matrix is $I_n - \frac{1}{n+1}J_n$, where $J_n$ is the matrix whose entries are all equal to $1$.
    Notice that, summing all the relations and dividing by $n+1$, we have $\sum_{i = 1}^n k_i = k_0$. Writing $p$ as $(2k_1 + k_2+...+k_n + j, k_1 + 2k_2 + k_3 + ... + k_n + j, ..., k_1 + ... + k_{n-1} + 2k_n + j)$, one can check that it belongs to the interior of the $n$-dimensional cell of $\overline{\tau_m^n}$ defined by the inequalities
    \begin{equation}
    \begin{cases}
        (n+1)(k_0 + j -1) < \sum_{i = 1}^n x_i < (n+1)(k_0 + j)\\
        -(n+1)(k_1-1) < -nx_1 + \sum_{i=2}^n x_i < -(n+1)k_1\\
        ...\\
        -(n+1)(k_n-1) < -nx_n + \sum_{i=1}^{n-1}n x_i < -(n+1)k_n
    \end{cases} \, .
    \end{equation}
    Indeed, the first line becomes $(n+1)(j-1) < nj < (n+1)j$ after substracting $(n+1)k_0 = \sum_{i=1}^n (n+1)k_i$, while the others all give $-(n+1) < -j < 0$.

    The second part of the lemma is then straightforward.
\end{proof}

    Now, we can prove Theorem \ref{prop_large_b0_alldim}.

\begin{proof}
    Let $n$ be a positive integer and let $m$ be an even positive integer. Assign the sign $+$ to every vertex of the subdivision $\overline{\tau_m^n}$. Refine $\overline{\tau_m^n}$ by taking all the cones with vertex an integer point $p$ lying in the interior of an $n$-dimensional cell $C_p$ of $\overline{\tau_m^n}$ and with basis a $(n-1)$-face of $C_p$. The obtained subdivision is not a triangulation if $n>3$. In that case, refine it into a convex triangulation by triangulating the cells that are not simplices. 
    We denote by $\tau_m^n$ the triangulation obtained this way. The vertices of $\tau_m^n$ that do not yet bear a sign are assigned the sign $-$. Now, we dilate $\tau_m^n$ by $2$ to obtain a triangulation ${\tau'}_m^{n}$ of $m\Pi^n$ and we denote by $A_m^n$ the $T^2$-hypersurface obtained by using combinatorial patchworking with this signed triangulation of $m\Pi^n$. Notice that the way we obtained $\tau_m^n$ from $\overline{\tau_m^n}$ ensures that the link of a vertex of ${\tau'}_m^{n}$ that bears the sign $-$ consists only of points bearing the sign $+$. Hence, in the star of each of the $2^n$ copies of every vertex of ${\tau'}_m^{n}$ that bears the sign $-$, there is a connected component homeomorphic to a sphere of the $PL$-hypersurface obtained during the patchworking process. It follows from Lemma \ref{lemma_points_in_cells} that the number of vertices of $\tau_m^n$ bearing the sign $-$ is asymptotically equivalent to $\frac{n}{n+1}l^*(\Pi_{\frac{m}{2}}^n)$. The result follows.
\end{proof}

\begin{rmk}
    The construction is most easily described for the toric variety associated to $\Pi^n$, but can be adapted to other ambient spaces, such as $\P^n$, to obtain families that satisfy the same asymptotical property.
\end{rmk}

\subsection{First Betti number}

Finally, we present a construction of a family of $T^2$-surfaces in $(\P^1)^3$. Once again, the ambient space is chosen so that the description of the construction and the computation of the invariants of the resulting surfaces are as easy as possible. However, the construction can be adapted to other ambient spaces, for example $\P^3$.

\begin{prop}
    \label{prop_big_b1}
    There exists a family $(B_m^3)_{m \in 2\N}$ of $T^2$-surfaces in $(\P^1)^3$ such that for any $m \in 2 \N$, the surface $B_m^3$ is of multidegree $(m,m,m)$ and
    \begin{equation*}
        \chi(\R A_m^3) = -\frac{18}{8} m^3 \customeqtrois -\frac{1}{2}h^{1,1}(\C A_m^3) - \frac{1}{4}m^3
    \customeqtrois \sigma(\C A_m^3) - \frac{1}{4}m^3.    
    \end{equation*}
    In particular, this means that $b_1(\R A_m^3) \customgeqtrois \frac{1}{2}h^{1,1}(\C A_m^3) + \frac{1}{4}m^3$.
\end{prop}

\begin{lem}
    \label{lem_big_b1}
    There exists a $T^2$-surface $B$ of multidegree $(2,2,2)$ in $(\P^1)^3$ such that $\chi(\R B) = -18$.
\end{lem}

\begin{proof}
    Consider the triangulation of the polytope in $\R^3$ with vertices $(0,0,0)$, $(2,0,0)$, $(0,2,0)$, $(2,2,0)$, $(0,0,2)$, $(2,0,2)$, $(0,2,2)$, $(2,2,2)$ given by the tetrahedra with vertices
    \begin{itemize}
        \item $(0,0,0)$, $(2,0,0)$, $(0,2,0)$, $(0,0,2)$;
        \item $(2,0,0)$, $(2,0,2)$, $(0,0,2)$, $(2,2,2)$;
        \item $(0,2,0)$, $(0,2,2)$, $(0,0,2)$, $(2,2,2)$;
        \item $(0,0,2)$, $(2,2,0)$, $(2,0,0)$, $(0,2,0)$;
        \item $(0,0,2)$, $(2,2,0)$, $(2,0,0)$, $(2,2,2)$
        \item $(0,0,2)$, $(2,2,0)$, $(0,2,0)$, $(2,2,2)$.
    \end{itemize}
    The vertices $(2,0,0)$, $(0,2,0)$ and $(2,2,2)$ receive the sign $-$, while the other receive the sign $+$.
    This triangaulation contains
    \begin{itemize}
        \item $6$ tetrahedra with vertices of opposite signs;
        \item $18$ triangles with vertices of opposite signs, among which $12$ lie on a $2$-face of the polytope;
        \item $12$ segments with vertices of opposite signs, among which $3$ lies on a $2$-face of the polytope and $9$ lie on a $1$-face of the polytope.
    \end{itemize}
    Hence the Euler characteristic of the patchworked surface is $8\times(6-6) + 4\times(-12+3) + 2 \times 9 = -18$.
\end{proof}

\begin{rmk}
    In this case, $-18$ is the lowest value allowed by the Comessatti inequalities (see \textit{e.g.} \cite{kharlamov}). The classification of real $K3$ surfaces (see \textit{e.g.} \cite{real_enriques}) shows that $\R B$ is actually homeomorphic to a connected orientable surface of genus $10$.
\end{rmk}

We can now prove Proposition \ref{prop_big_b1}.

\begin{proof}
    The desired triangulation is simply a gluing of some copies of the image under the reflections with respect to all coordinate hyperplane of the triangulated cube described in the proof of \ref{lem_big_b1}.

    Choose a positive integer $k$ and consider the polytope $\Pi_{k,1,1}$ with vertices $(0,0,0)$, $(2k,0,0)$, $(0,2,0)$, $(2k,2,0)$, $(0,0,2)$, $(2k,0,2)$, $(0,2,2)$, $(2k,2,2)$. We denote by $x$, $y$, $z$ the standard coordinates in $\R^3$. Subdivide $\Pi_{k,1,1}$ by the hyperplanes with equation $x = 2i$ for any integer $i \in \{1,...,k-1 \}$. 

    For any even integer $i \in \{0,...,k-1 \}$, consider the signed triangulation of the cube with vertices $(2i,0,0)$, $(2i,2,0)$, $(2i+2,0,0)$, $(2i+2,2,0)$, $(2i,0,2)$, $(2i,2,2)$, $(2i+2,0,2)$, $(2i+2,2,2)$ obtained by translating the signed triangulation described in the proof of \ref{lem_big_b1} along the vector $(2i,0,0)$.

    For any odd integer $i \in \{0,...,k-1 \}$, consider the signed triangulation of the cube with vertices $(2i,0,0)$, $(2i,2,0)$, $(2i+2,0,0)$, $(2i+2,2,0)$, $(2i,0,2)$, $(2i,2,2)$, $(2i+2,0,2)$, $(2i+2,2,2)$ obtained by applying the reflection with respect to the plane $\{ x=0 \}$ and then translating along the vector $(2i+2,0,0)$ the signed triangulation described in the proof of Lemma \ref{lem_big_b1}.

    The real part of a resulting multidegree $(2k,2,2)$ patchworked surface is obtained from $k$ copies of the real part of the surface obtained in Lemma \ref{lem_big_b1} that are cut and glued along the disjoint union of some smooth closed curves. Hence, the Euler characteristic of the real part of the surface is simply $-18k$.

    In a similar fashion, we can use the signed triangulation of $\Pi_{k,1,1}$ we described to construct a signed triangulation of $\Pi_{k,k,1}$ such that the Euler characteristic of the real part of an associated patchworked surface is equal to $-18k^2$. This signed triangulation can in turn be used to construct a signed triangulation of $\Pi_{k,k,k}$ such that the Euler characteristic of the real part of an associated patchworked surface is equal to $-18k^3$.
\end{proof}




\newpage
\bibliographystyle{alpha}
\bibliography{biblio}

\end{document}